\documentclass[reqno,12pt,centertags]{amsart}
\usepackage{amsmath,amsthm,amscd,amssymb,latexsym,upref,stmaryrd}
\usepackage[cp1251]{inputenc}
\usepackage[english]{babel}
\usepackage{times}
\usepackage{amsfonts}
\usepackage[numbers,sort&compress]{natbib}
\usepackage{hyperref}
\newcommand*{\mailto}[1]{\href{mailto:#1}{\nolinkurl{#1}}}
\usepackage{amsmath}
\usepackage{amsthm}

\newtheorem{theorem}{Theorem}[section]
\newtheorem{lemma}{Lemma}[section]

\newtheorem{proposition}{Proposition}[section]
\newtheorem{ip}{Inverse Problem}[section]

\numberwithin{equation}{section}
\begin{document}

\thispagestyle{empty}

\noindent{\large\bf  Local solvability and stability of the generalized inverse Robin-Regge problem with complex coefficients}
\\

\noindent {\bf  Xiao-Chuan Xu}\footnote{School of Mathematics and Statistics, Nanjing University of Information Science and Technology, Nanjing, 210044, Jiangsu,
People's Republic of China, {\it Email:
xcxu@nuist.edu.cn}} ,
\noindent {\bf Natalia Pavlovna Bondarenko}\footnote{Department of Applied Mathematics and Physics, Samara National Research University,
Moskovskoye Shosse 34, Samara 443086, Russia,}
\footnote{Department of Mechanics and Mathematics, Saratov State University,
Astrakhanskaya 83, Saratov 410012, Russia,
Email: {\it bondarenkonp@info.sgu.ru}
}
\\

\noindent{\bf Abstract.}
{We prove local solvability and stability of the inverse Robin-Regge problem in the general case, taking eigenvalue multiplicities into account. We develop the new approach based on the reduction of this inverse problem to the recovery of the Sturm-Liouville potential from the Cauchy data.}

\medskip
\noindent {\it Keywords: }{Robin-Regge problem; Cauchy data; Inverse spectral problem; Local solvability; Stability; Multiple eigenvalues}

\medskip
\noindent{\it 2010 Mathematics Subject Classification:} 34A55; 34B08; 34L40; 35R30

\section{Introduction}

Consider the following generalized Robin-Regge problem $L(q,h,\alpha,\beta)$:
\begin{equation}\label{1}
{-y^{\prime \prime}(x)+q(x)y(x)=\lambda^2y(x), \quad 0<x<a},
\end{equation}
\begin{equation}\label{2}
y'(0)-hy(0)=0,
\end{equation}
\begin{equation}\label{3}
y^{\prime}(a)+(i\lambda\alpha+\beta)y(a)=0,
\end{equation}
where $\lambda$ is spectral parameter, the complex-valued potential $q$ belongs to $L^2(0,a)$, $h,\beta\in \mathbb{C}$ and $\alpha>0$.

The problem $L(q,h,\alpha,\beta)$ arises in various models of mathematical physics, such as the problem of small transversal vibrations of a smooth inhomogeneous string subject to viscous damping \cite{MP,MP1}, the resonance scattering problem \cite{SB}, and the problem of determining the sharp of human vocal tract \cite{AMS}.

This paper is concerned with the inverse spectral problem that consists in recovery of the potential $q(x)$ and the coefficients of the boundary conditions \eqref{2}-\eqref{3} from the eigenvalues of $L(q, h, \alpha, \beta)$. In the theory of inverse spectral problems, the most complete results were obtained for operators induced by the Sturm-Liouville equation~\eqref{1} with boundary conditions independent of the spectral parameter (see the monographs \cite{VM, Lev, PT, FY} and references therein). In particular, Borg~\cite{BO} has proved that the real-valued potential $q$ is uniquely specified by the two spectra $\{ \lambda_{n,\nu} \}$, $\nu = 0, 1$, of the problems $L_\nu(q, h)$ given by \eqref{1}-\eqref{2} and the boundary condition $y^{(\nu)}(a) = 0$, $\nu = 0, 1$. Moreover, Borg~\cite{BO} obtained local solvability and stability of this inverse problem. Recently, the results of Borg were generalized by Buterin and Kuznetsova \cite{BK} to the case of the complex-valued potential $q$. The latter case is more difficult for investigation, since the spectra $\{ \lambda_{n,\mu} \}$ can contain multiple eigenvalues, which can split under a small perturbation.

However, the presence of the spectral parameter $\lambda$ in the boundary condition causes a significant qualitative difference of problem~\eqref{1}-\eqref{3} from the classical Sturm-Liouville problems. Namely, in order to recover the potential $q$ of the problem $L(q, h, \alpha, \beta)$, one needs only one spectrum instead of two spectra. This can be easily shown by the reduction of the inverse Robin-Regge problem to the Borg inverse problem by two spectra (see, e.g., \cite{X}). Nevertheless, the method of reduction to the Borg problem is inconvenient for studying various issues of the inverse problem theory, in particular, of local solvability and stability of the inverse problem. Therefore, the Regge-type problems require development of new methods for their investigation.

Some aspects of the inverse Regge-type problems were studied in the earlier papers \cite{RT1, KN1, KN2}. Important advances in the theory of the Regge-type problems have been achieved by Yurko \cite{Yur}, who considered various types
of inverse problems with linear and also with polynomial dependence on the spectral parameter in the boundary conditions. For the problem $L(q,h,\alpha,\beta)$ with real coefficients,  M\"{o}ller and Pivovarchik \cite{MP} proved the uniqueness and existence of the inverse problem solution.
In \cite{PV}, the Dirichlet-Regge inverse problem was studied with the boundary condition \eqref{2} replaced by $y(0) = 0$. Later on, Xu \cite{X} considered the problem $L(q,h,\alpha,\beta)$ with complex coefficients, where the uniqueness theorems are proved with reconstruction algorithms being provided.  In addition, Xu \cite{X} studied local solvability and stability of the inverse problem under some restrictions on eigenvalue perturbations.

In this paper, we suggest a new approach to the inverse Regge-type problems. We reduce the inverse Robin-Regge problem to the recovery of the Sturm-Liouville potential by the so-called Cauchy data, by using the special exponential Riesz basis. The ideas of this approach appeared in the papers by Bondarenko \cite{BP2, BP3}. Our method is convenient for dealing with multiple eigenvalues. As it was pointed out in \cite{KT}, our approach, in fact, provides the first constructive algorithm for interpolation of the Weyl function by its values in a countable set of points. We also mention that the reduction to the inverse problem by the Cauchy data has been recently applied to the inverse transmission eigenvalue problem by Buterin et al \cite{BCK}.

The main result of this paper is the following theorem on the local solvability and stability of the inverse Robin-Regge problem.
Denote
$$\mathbb{Z}_0=\mathbb{Z},\quad  \mathbb{Z}_1=\mathbb{Z}\setminus\{0\},\quad \mathbb{Z}_j^-=\mathbb{Z}_j\setminus\{1\}, \quad  j=0,1.$$
 It was known \cite{MP,X} that the eigenvalues, which can be denoted by $\{\lambda_n\}_{n\in \mathbb{Z}_j}$,  of the problem $L(q,h,\alpha,\beta)$ with $(-1)^{j+1}(\alpha-1)<0$ have the following asymptotics
  \begin{equation}\label{2.7}
\lambda_{n}=\frac{(|n|-\frac{j+1}{2})\pi}{a} {\rm sgn} n+\frac{i}{2 a} \ln \left|\frac{\alpha+1}{1-\alpha}\right|+\frac{P}{n}+\frac{\gamma_{j,n}}{n},\quad j=0,1,
\end{equation}
where $\{\gamma_{j,n}\}\in l_2$, and
\begin{equation}\label{nn}
P=\frac{1}{\pi}\left(\omega-\frac{\beta}{\alpha^{2}-1}\right),\quad \omega=h+\frac{1}{2}\int_0^aq(s)ds.
\end{equation}
In our notations, we agree that $j=0$ corresponds to the case $\alpha>1$ and $j=1$, to the case $\alpha<1$.
Consider the inverse problem that consists in recovery of $q$, $h$ and $\alpha\ne1$ from the known $\beta$ and the set $\{ \lambda_n \}_{n \in \mathbb Z_j^-}$ of all the eigenvalues except one. Note that the numeration of the eigenvalues is not uniquely fixed by the asymptotics \eqref{2.7}, so every eigenvalue can be excluded.

\begin{theorem}\label{th4.1}
Let $\{\lambda_n\}_{n\in \mathbb{Z}_j^{-}}$ $(j=0,1)$ be the eigenvalues  of the problem $L(q,h,\alpha,\beta)$ with complex-valued $q\in L^2(0,a)$, $h,\beta\in \mathbb{C}$ and $(-1)^{j+1}(\alpha-1)<0$. Then there exists $\varepsilon>0 $ (depending on the problem $L(q,h,\alpha,\beta)$) such that for any sequence $\{\tilde{\lambda}_n\}_{n\in \mathbb{Z}_j^-}$ satisfying
  \begin{equation}\label{4.1}
 \Lambda:= \sqrt{\sum_{n\in \mathbb{Z}_j^-}(n^2+1)|\lambda_n-\tilde{\lambda}_n|^2}\le\varepsilon,
  \end{equation}
there exist  unique $\tilde{q}\in L^2(0,a)$ and $\tilde{h}\in \mathbb{C}$ such that  $\{\tilde{\lambda}_n\}_{n\in \mathbb{Z}_j^-}$ are the eigenvalues  of the problem $L(\tilde{q},\tilde{h},{\alpha},{\beta})$. Moreover,
 \begin{equation}\label{jia}
   \|\tilde{q}-q\|_{L^2} \le C \Lambda,\quad |\tilde{h}-h|\le C \Lambda,
 \end{equation}
 where $C>0$ depends only on the problem $L(q,h,\alpha,\beta)$.
\end{theorem}

An important difference of this theorem comparing with the results of \cite{X} is that in \cite{X} the following stability estimates are obtained:
\begin{equation*}
   \|\tilde{q}-q\|_{L^2}<C \Lambda^{1/p},\quad |\tilde{h}-h|\le C \Lambda^{1/p}
 \end{equation*}
with the additional constant $p\ge1$ depending on $q(x)$ and $h$. Moreover, the proofs in \cite{X} contain a mistake related with eigenvalue multiplicities. In fact, the results of \cite{X} are valid only in the special case when the multiplicities of $\{ \tilde \lambda_n \}$ coincide with the multiplicities of $\{ \lambda_n \}$. But under a small perturbation, multiple eigenvalues of the problem $L(q, h, \alpha, \beta)$ can split into smaller groups. In the present paper, we take this effect
into account and prove Theorem~\ref{th4.1} in the general case, without any restrictions on the eigenvalue multiplicities. Moreover, our new method allows us to obtain the improved estimate~\eqref{jia} without $p$.

The paper is organized as follows. In Section~2, we provide the definition of the Cauchy data and prove the local solvability and stability of the inverse problem by the Cauchy data (Theorem~\ref{thca}). This theorem plays an auxiliary role in this paper, but also can be considered as a separate result. In Section~3, the proof of the main Theorem~\ref{th4.1} is provided.

\section{Inverse problem by the Cauchy data}

In this section, we  prove an auxiliary theorem on the local solvability and stability of the inverse problem by the Cauchy data.

Let $\varphi(x,\lambda)$ be the solution of (\ref{1}) with the initial values $\varphi(0,\lambda)=1,\varphi'(0,\lambda)=h$.
It is well known that
\begin{equation}\label{z1}
  \varphi(x,\lambda)=\cos (\lambda x)+\int_{0}^{x}K(x,t)\cos(\lambda t)dt,
\end{equation}
where $K(x,t)$ is a two variable continuous function with first partial derivatives, satisfying $K_t(a,\cdot),K_x(a,\cdot)\in L^2(0,a)$, and $  K(a,a)=\omega.$
Using (\ref{z1}),  we have
\begin{equation}\label{2.1}
  \varphi(a,\lambda)=\cos (\lambda a)+\frac{\omega\sin(\lambda a)}{\lambda}-\int_{0}^aK_t(a,t)\frac{\sin (\lambda t)}{\lambda}dt,
\end{equation}
\begin{equation}\label{2.2}
  \varphi'(a,\lambda)=-\lambda\sin (\lambda a)+{\omega\cos(\lambda a)}+\int_{0}^aK_x(a,t)\cos (\lambda t)dt.
\end{equation}
 The set $\{K_t(a,t), K_x(a,t),\omega\}$ is called the Cauchy data for $q$ and $h$. 
 We shall consider the following inverse problem.

 \begin{ip}\label{ipc}
 Given the Cauchy data $\{K_t(a,t), K_x(a,t),\omega\}$, find the potential $q(x)$ and $h$.
 \end{ip}
 
We remark here, when $q$ and $h$ are real, Rundell and Sacks \cite{RS1} gave the numerical reconstruction algorithm for  Problem \ref{ipc}, and applied the technique to the inverse resonance problem \cite{RS}. We shall consider the local solvability and stability for Problem \ref{ipc} with complex  $q$ and $h$.

\begin{theorem}\label{thca}
Let $q(x)$ be a fixed complex-valued function from $L^2(0,a)$, and let $h\in \mathbb{C}$ be a fixed number. Denote by $\{K_1,K_2,\omega\}$  the corresponding Cauchy data. Then there exists $\varepsilon>0$ (depending only on $q$ and $h$) such that, for any functions $\{\tilde{K}_1,\tilde{K}_2\}$ satisfying
\begin{equation}\label{cau}
\Xi:=\max\{\|\tilde{K}_1-K_1\|_{L^2(0,a)},\|\tilde{K}_2-K_2\|_{L^2(0,a)}\}\le \varepsilon,
\end{equation}
there exists a unique  function $\tilde{q}\in L^2 (0,a)$ such that $\{\tilde{K}_1,\tilde{K}_2,\omega\}$ are the Cauchy data for $\tilde{q}$ and $\tilde{h}=\omega-\frac{1}{2}\int_0^a\tilde{q}(x)dx$. Moreover,
\begin{equation}\label{cau1}
 \|\tilde{q}-q\|_{L^2(0,a)} \le C\Xi,\quad |\tilde{h}-h| \le C\Xi,
\end{equation}
 where $C$ depends only on $q$ and $h$.
\end{theorem}

Note that the analog of Theorem~\ref{thca} for the case of the Dirichlet boundary condition $y(0) = 0$ was proved in \cite{BP2}.

\begin{proof}
Let us prove Theorem \ref{thca} by showing several  auxiliary propositions.  Let $z=\lambda^2$. Define the functions
\begin{equation}\label{ca2.1}
 \eta_1(z):=\cos (\lambda a)+\frac{\omega\sin(\lambda a)}{\lambda}-\int_{0}^aK_1(t)\frac{\sin (\lambda t)}{\lambda}dt,
\end{equation}
\begin{equation}\label{ca2.2}
  \eta_2(z):=-\lambda\sin (\lambda a)+{\omega\cos(\lambda a)}+\int_{0}^aK_2(t)\cos (\lambda t)dt.
\end{equation}
 By the standard method related to the Rouch\'{e}'s theorem, one can easily obtain the asymptotics of the zeros of the function $\eta_2(z)$.
\begin{proposition}
 Let $K_2(t)$ be  an arbitrary complex-valued function in $L^2{(0,a)}.$ Then the zeros $\{z_n\}_{n\ge0}$ with $|z_{n+1}|\ge |z_n|$ of the function $\eta(z)$
have the asymptotics
\begin{equation}\label{ca1}
\rho_n:=  \sqrt{z_n}=\frac{n\pi}{a}+O\left(\frac{1}{n}\right).
\end{equation}
\end{proposition}

In view of the asymptotic formula (\ref{ca1}), we can find the smallest integer $n_1 \ge1$ such that the zeros $\{z_n\}_{n\ge n_1}$ are simple and $|z_{n_1}|>|z_{n_1-1}|$. Consider the disk $ \Gamma_0=\{z:|z|\le {(|z_{n_1}|+|z_{n_1-1}|)}/{2}\}$. Obviously, the zeros $\{z_n\}_{n=0}^{n_1-1}\subset {\rm int} \, \Gamma_0$, and the zeros $\{z_n\}_{n\ge n_1}$ lie strictly outside $\Gamma_0$.

Denote by $k_n$ the multiplicity of the value $z_n$  in the sequence $\{z_n\}_{n\ge0}$, and assume that multiple $z_n$'s are neighboring: $z_n=z_{n+1}=\cdot\cdot\cdot=z_{n+k_n-1}$.  Define $I_0:=\{n\ge1,z_n\ne z_{n-1}\}\cup\{0\}.$ Introduce the Weyl function $M(z)$ and the sequence $\{M_n\}_{n\ge0}$ as follows:
\begin{equation*}
 M(z):=\frac{\eta_1(z)}{\eta_2(z)},\quad M_n:=\mathop{\rm Res}\limits_{z=z_n}(z-z_n)^{v}M(z),\quad n\in I,\quad v=0,1,...,k_n-1.
\end{equation*}

In the following discussion, we agree that, if a certain symbol $\gamma$ denotes an object constructed by $\{K_1, K_2, \omega\}$, then the symbol $\tilde{\gamma}$ with tilde denotes the analogous object constructed by $\{\tilde{K}_1, \tilde{K}_2, \omega\} .$

\begin{lemma} \label{leap}
Let $K_1,K_2$ be fixed complex-valued functions in $L^{2}(0, a),$ and let $\omega \in \mathbb{C} .$ Then, there exists $\varepsilon>0$ (depending on $K_1, K_2, \omega)$ such that, for any $\tilde{K}_1, \tilde{K}_2 \in L^{2}(0, \pi)$ satisfying \eqref{cau}, the zeros $\left\{\tilde{z}_{n}\right\}_{n=0}^{n_{1}-1}$ of $\tilde{\eta}_2(z)$ lie strictly inside $\Gamma_{0}$ and

\begin{equation}\label{cau2}
\max _{z\in \partial\Gamma_{0}}|M(z)-\tilde{M}(z)| \leq C \Xi.
\end{equation}
For $n \geq n_{1},$ we have $\tilde{k}_{n}=1$ and
\begin{equation}\label{cau3}
\left(\sum_{n=n_{1}}^{\infty}\left(n \xi_{n}\right)^{2}\right)^{1 / 2} \leq C{\Xi}
\end{equation}
where $\xi_{n}:=\left|\rho_{n}-\tilde{\rho}_{n}\right|+ \left|M_{n}-\tilde{M}_{n}\right| .$ Here the positive constant $C$ in \eqref{cau2} and \eqref{cau3} depends only on $K_1$, $K_2$, and $\omega$.
\end{lemma}
\begin{proof}
In the proof, we denote by $C_i$ ($i=1,...,20$) positive constants, which depend only on $K_1$, $K_2$, and $\omega$.
From the conditions of the lemma, we see that
\begin{equation}\label{cau4}
|\eta_2(z)|\ge C_1,\quad |\eta_2(z)-\tilde{\eta}_2(z)|\le C_2\Xi,\quad z\in \partial\Gamma_0.
\end{equation}
 It follows that
 \begin{equation}\label{cau5}
 |\eta_2(z)-\tilde{\eta}_2(z)|/|\eta_2(z)|<1,\quad z\in\partial\Gamma_0,
\end{equation}
for sufficiently small $\varepsilon>0.$ Thus, we have from the Rouch\'{e}'s theorem that the function $\tilde{\eta}_2(z)$ has the same number of zeros as $\eta_2(z)$ inside $\Gamma_0$.  According to our notations, these zeros of $\tilde{\eta}_2(z)$ are $\left\{\tilde{z}_{n}\right\}_{n=0}^{n_{1}-1}$. Again, using \eqref{cau5}, we have
\begin{equation}\label{cau6}
  |\tilde{\eta}_2(z)|\ge |\eta_2(z)|-|\eta_2(z)-\tilde{\eta}_2(z)|\ge C_3,\quad z\in \partial\Gamma_0
\end{equation}
for sufficiently small $\varepsilon>0.$
Using the definition of $M(z)$ together with (\ref{cau4}) and (\ref{cau6}), and noting that $\eta_i(z)$ ($i=1,2$) are bounded on $\Gamma_0$, we obtain
\begin{equation*}
 | M(z)-\tilde{M}(z)|\le\frac{|\eta_1(z)-\tilde{\eta}_1(z)||\eta_2(z)|+|\eta_2(z)-\tilde{\eta}_2(z)||{\eta}_1(z)|}{|\eta_2(z)\tilde{\eta}_2(z)|}\le C\Xi,\quad z\in\partial\Gamma_0,
\end{equation*}
which implies (\ref{cau2}).

Now, let us prove \eqref{cau3}. We shall first prove the inequality for the part of $|\rho_n-\tilde{\rho}_n|.$ For $n\ge n_1,$  consider the disks $\gamma_{n,\delta}:=\{\lambda:|\lambda-\rho_n|\le\delta\}$, where $\delta>0$ is fixed and so small that $\delta\le \frac{|\rho_n-\rho_{n+1}|}{2}$ for all $n\ge n_1.$ Then the function $\eta_2(\lambda^2)$ has exactly one zero $\rho_n\in  {\rm int}\, \gamma_{n,\delta}$ in the $\lambda$-plane for every $n\ge n_1.$  It follows from \eqref{ca2.2} that
\begin{equation}\label{cau7}
|\eta_2(\lambda^2)|\le n C_4,\quad  \lambda\in \gamma_{n,\delta},\quad|\dot{\eta}_2(\rho_n^2)|\ge nC_5,\quad n\ge n_1
\end{equation}
where $\dot{\eta}_2(\lambda^2):=\frac{d\eta_2(\lambda^2)}{d\lambda}$.  For $\lambda\in {\rm int }\, \gamma_{n,\delta}$, we have the Taylor formula
\begin{equation}\label{cau9}
  \eta_{2}(\lambda^2)=\eta_{2}\left(\rho_n^2\right)+\dot{\eta}_2(\rho_{n}^2)\left(\lambda-\rho_{n}\right)+\frac{\left(\lambda-\rho_{n}\right)^{2}}{2 \pi {i}} \int_{\partial\gamma_{n, \delta}} \frac{\eta_{2}(\rho^2) \mathrm{d} \rho}{\left(\rho-\rho_{n}\right)^{2}(\rho-\lambda)}.
\end{equation}
Using (\ref{cau7}) and (\ref{cau9}), we obtain
\begin{equation}\label{cau10}
 |  \eta_{2}(\lambda^2)|\ge nC_5|\lambda-\rho_{n}|-\frac{nC_4}{\delta^2(\delta-\delta_1)}|\lambda-\rho_{n}|^2\ge nC_6 |\lambda-\rho_{n}|,\quad \lambda\in \gamma_{n,\delta_1},
\end{equation}
where $\delta_1\in (0,\delta)$ is  sufficiently small and fixed.

For sufficiently small $\varepsilon>0$, we have
\begin{equation}\label{cau8}
 |\eta_2(\lambda^2)-\tilde{\eta}_2(\lambda^2)|\le C_7\Xi,\quad \lambda\in \partial\gamma_{n,\delta_1},\quad n\ge n_1.
\end{equation}
Using (\ref{cau8}), and noting $|\eta_2(\lambda^2)|\ge C_8 $ for $\lambda\in \partial\gamma_{n,\delta_1}$ for $ n\ge n_1$, we obtain that for sufficiently small $\varepsilon>0$ there holds
\begin{equation*}
   |\eta_2(\lambda^2)-\tilde{\eta}_2(\lambda^2)|<|\eta_2(\lambda^2)|,\quad \lambda\in \partial\gamma_{n,\delta_1},\quad n\ge n_1.
\end{equation*}
It follows from  the Rouch\'{e}'s theorem that the function $\tilde{\eta}_2(\lambda^2)$ has exactly one zero  $\tilde{\rho}_n\in{\rm int}\gamma_{n,\delta_1}$ for each $n\ge n_1$. Using \eqref{cau10} and (\ref{ca2.2}), we get
\begin{equation}\label{cau11}
 |\tilde{\rho}_n-\rho_{n}|\le \frac{1}{nC_6}|  \eta_{2}(\tilde{\rho}_n^2)-\tilde{\eta}_{2}(\tilde{\rho}_n^2)|\le \frac{C_9}{n}\left|\int_0^a\hat{K}_2(t)\cos (\tilde{\rho}_nt)dt\right|,
\end{equation}
where $\hat{K}_2:=\tilde{K}_2-K_2$. Using (\ref{cau}) and the asymptotic formula \eqref{ca1} of $\tilde{\rho}_n$, we have
\begin{equation}\label{cau12}
  \begin{aligned}
\left|\int_{0}^{a} \hat{K}_2(t) \cos \left(\tilde{\rho}_{n} t\right) \mathrm{d} t\right| & \leq\left|\int_{0}^{a} \hat{K}_2(t) \cos\frac{n\pi t}{a} \mathrm{d} t\right|+\left|\int_{0}^{a} \hat{K}_2(t)\!\!\left(\cos \left(\tilde{\rho}_{n} t\right)\!-\!\cos\frac{n\pi t}{a}\right)\! \mathrm{d} t\right| \\
& \leq\left|\hat{K}_{2,n}\right|+\frac{C_{10} \Xi}{n}, \quad n \geq n_{1}, \quad \hat{K}_{2,n}:=\int_{0}^{a} \hat{K}_2(t) \cos \frac{n\pi t}{a} \mathrm{d} t.
\end{aligned}
\end{equation}
It follows from \eqref{cau11} and \eqref{cau12} that
\begin{equation*}
 n|\tilde{\rho}_n-\rho_{n}|\le C_9 \left|\hat{K}_{2,n}\right|+\frac{C_{11} \Xi}{n}.
\end{equation*}
Using the Bessel inequality for the Fourier coefficients $\{K_{2,n}\}_{n\ge n_1}$ together with (\ref{cau}), we have
\begin{equation}\label{cau13}
\sqrt{\sum_{n=n_1}^\infty n^2|\tilde{\rho}_n-\rho_{n}|^2}\le C_{12}\Xi.
\end{equation}

Let us prove the inequality (\ref{cau3}) for the part of $|M_n-\tilde{M}_n|$. Note that $\{z_n\}_{n\ge n_1}$ are simple zeros of $\eta_2(z)$. Thus we have
\begin{equation}\label{cau14}
 M_n:=\mathop{\rm Res}\limits_{z=z_n}M(z)=\frac{\eta_1(z_n)}{{\eta}_2'(z_n)},\quad n\ge n_1.
\end{equation}
For the sufficiently small $\varepsilon>0$,  the analogous relation is valid for $\tilde{M}_n$ for $n\ge n_1$. Thus we have
\begin{equation}\label{cau15}
 \tilde{M}_{n}-M_{n}=\frac{\left(\tilde{\eta}_{1}(\tilde z_n)-\eta_{1}(z_n)\right) {\eta}_{2}'(z_n)+\eta_{1}(z_n)\left({\eta}_{2}'(z_n)-{\tilde{\eta}}_{2}'(\tilde z_n)\right)}{{\eta}_{2}'(z_n) {\tilde{\eta}}_{2}'(\tilde z_n)}, \quad n \geq n_{1}
\end{equation}
From (\ref{ca2.1}) and (\ref{ca2.2}), we know that
\begin{equation}\label{cau16}
\begin{split}
 &|{\eta}_1(z_n)|\le C_{13},\quad |{{\eta}}_2'(z_n)|\ge C_{14},\quad  |\eta_1(z_n)-\tilde{\eta}_1(\tilde z_n)|\le C_{15}\left(\frac{|\hat{K}_{1,n}|}{n}+\frac{\Xi}{n^2}\right)\\
 & |{\tilde{\eta}}_2'(\tilde z_n)|\ge C_{16},\quad |{\eta}_2'(z_n)-{\tilde{\eta}}_2'(z_n)|\le C_{17}\left(\frac{|\check{K}_{2,n}|}{n}+\frac{\Xi}{n^2}\right),\quad n\ge n_1,
\end{split}
\end{equation}
where
\begin{equation*}
  \hat{K}_{1,n}=\int_0^a [\tilde{K}_1(t)-K_1(t)]\sin \frac{n\pi t}{a}dt,\quad   \check{{K}}_{2,n}=\int_0^a t[K_2(t)-\tilde{K}_2(t)]\sin \frac{n\pi t}{a}dt.
\end{equation*}
Using (\ref{cau15}), (\ref{cau16}) and the second inequality in (\ref{cau7}), we have
\begin{equation}\label{cau17}
  |\tilde{M}_n-M_n|\le \frac{C_{18}(|\hat{K}_{1,n}|+|\check{K}_{2,n}|)}{n}+\frac{C_{19}\Xi}{n^2}.
\end{equation}
Similarly to (\ref{cau13}), we get
\begin{equation}\label{cau18}
 \sqrt{\sum_{n=n_1}^\infty n^2 |\tilde{M}_n-M_n|^2}<C_{20}\Xi.
\end{equation}
Together with (\ref{cau13}) and (\ref{cau18}), we arrive at (\ref{cau3}). The proof of Lemma \ref{leap} is complete.
\end{proof}

In \cite{BP1,BS}, the following inverse problem is considered.

\begin{ip}\label{ip2}
Given the data $\left\{{z}_{n}, {M}_{n}\right\}_{n=0}^{\infty}$, find $q$ and $h$.
\end{ip}
In \cite{BP1}, Bondarenko proved the local solvability and stability for the above Inverse Problem \ref{ip2}.

\begin{proposition}\label{proa}
 Let $q \in L^{2}(0, a)$ and $h\in \mathbb{C}$ be fixed. Then, there exists $\varepsilon>0$ (depending on $q$ and $h$) such that, for any complex numbers $\left\{\tilde{z}_{n}, \tilde{M}_{n}\right\}_{n=0}^{\infty}$ satisfying the estimate
$$
\Omega:=\max \left\{\max _{\lambda \in \partial\Gamma_{0}}|M(\lambda)-\tilde{M}(\lambda)|,\left(\sum_{n=n_{1}}^{\infty}\left(n \xi_{n}\right)^{2}\right)^{1 / 2}\right\} \leq \varepsilon
$$
there exist the unique complex-valued function $\tilde{q} \in L^{2}(0, a)$ and $\tilde{h}\in\mathbb{C}$ being the solution of Inverse Problem \ref{ip2} for $\left\{\tilde{z}_{n}, \tilde{M}_{n}\right\}_{n=0}^{\infty} .$ Moreover,
\begin{equation*}
  \|\tilde{q}-q\|_{L^2(0,a)}\le C \Omega,\quad |\tilde{h}-h|\le C \Omega,
\end{equation*}
where the constant $C$ depends only on $q $ and $h.$
\end{proposition}

Using Lemma \ref{leap} and Proposition \ref{proa}, we finish the proof of Theorem \ref{thca}.
\end{proof}

\section{Proof of Theorem \ref{th4.1}}

In this section, we prove Theorem~\ref{th4.1} on the local solvability and stability theorem for the inverse Robin-Regge problem.

Note that the eigenvalues of the problem $ L(q,h,\alpha,\beta)$ coincide with the zeros of the \emph{characteristic function}
\begin{equation}\label{2.0}
  \Delta(\lambda)=\varphi'(a,\lambda)+(i\lambda \alpha+\beta)\varphi(a,\lambda).
\end{equation}
Denote by $m_k$ the multiplicity of the value $\lambda_k$  in the sequence $\{\lambda_n\}_{n\in \mathbb{Z}_j^-}$. In view of the asymptotics (\ref{2.7}),  there are at most finitely many multiple eigenvalues. Therefore, $m_n = 1$ for all  $|n|\ge n_0$ for some $n_0>0$.
Define the set
$$
\mathcal{S}_j:=\{n\in \mathbb{Z}_j^{-}\colon \lambda_n\ne\lambda_{k},\forall k\in \mathbb{Z}_j^{-} \colon k < n \},\quad j=0,1.
$$
Clearly, the sequence $\{\lambda_n\}_{n\in \mathcal{S}_j}$ consists of elements of $\{\lambda_n\}_{n\in \mathbb{Z}_j^-}$ being taken only once.

Without loss of generality, impose the following assumption.

\medskip

\textbf{Assumption} ($\mathcal N$): The multiple values  $\lambda_n$ in the sequence $\{ \lambda_n \}_{n \in \mathbb Z_j^-}$ are neighboring: $\lambda_n=\lambda_{n + 1}=\dots=\lambda_{n + m_n - 1}$ for all $n \in \mathcal S_j$.

\medskip

Introduce the functions
 \begin{equation}\label{s1}
 u_{n + \nu}(t):=(it)^\nu e^{i\lambda_nt},\quad n\in \mathcal{S}_j,\quad \nu=0,1,\dots,m_n-1.
\end{equation}

\begin{lemma}[See \cite{X}]\label{l2.2}
Suppose that the sequence $\{\lambda_n\}_{n\in \mathbb{Z}_j^-}$ $(j=0,1)$ satisfies the asymptotics (\ref{2.7}) and assumption ($\mathcal N$). Then the system $\{u_n(t)\}_{n\in \mathbb{Z}_j^-}$ is a Riesz basis in $L^2(-a,a)$.
\end{lemma}

Define the inner product in $L^2(-a,a)$ as
 $$
(g_1,g_2):=\int_{-a}^a \overline{g_1(t)}g_2(t)dt,\quad \forall g_1,g_2\in L^2(-a,a).
$$
Substituting (\ref{z1}), (\ref{2.1}) and (\ref{2.2}) into (\ref{2.0}), we have
 \begin{align}\label{x6}
 \Delta(\lambda)=f(\lambda)+\frac{1}{2}\int_{-a}^a[\overline{M(t)}+\alpha \overline{N(t)}]e^{i\lambda t}dt=\frac{1}{2}(M(t)+\alpha N(t),e^{i\lambda t}),
 \end{align}
 where
  \begin{equation}\label{x6s}
   f(\lambda):=-\lambda[\sin (\lambda a)-i\alpha\cos(\lambda a)]+(\omega+\beta)\cos(\lambda a)+i\alpha\omega\sin(\lambda a),
   \end{equation}
 \begin{equation}\label{x7}
   \overline{M(t)}=\left\{\begin{split}
                &K_x(a,t)+\beta K(a,t),\quad t\in(0,a),\\
               & K_x(a,-t)+\beta K(a,-t),\; t\in(-a,0),
                \end{split}\right.
 \end{equation}
  \begin{equation}\label{x8}
\overline{N(t)}=\left\{\begin{split}
                &-K_t(a,t),\quad t\in(0,a),\\
               & K_t(a,-t),\quad t\in(-a,0).
                \end{split}\right.\quad
 \end{equation}
 It is obvious that $M(t)$ is even and $N(t)$ is odd. Denote
 \begin{equation}\label{x9}
 w_{n + \nu}:=-f^{(\nu)}(\lambda_n),\quad  n\in \mathcal{S}_j,\quad \nu=0,1,\dots, m_n-1.
\end{equation}
Then we have
\begin{equation}\label{x10}
  \frac{1}{2}\left({M}+\alpha{N},u_n\right)=w_n,\quad n\in \mathbb{Z}_j^-.
\end{equation}

To deal with the multiple eigenvalues, we need the following lemma from~\cite{MW}.
\begin{lemma}\label{l2.3}
  Assume that $f(z)$ is an entire function, and $z_1$,..., $z_m$ (not necessarily distinct) are in the disk $\{z\colon |z-z_0|\le r<1/2\}$. Let $p(z)$ be the unique polynomial of degree at most $m-1$ which interpolates $f(z)$ and its derivatives in the usual way at the points $z_j$, $j={1,...,m}$: namely, if $z_j$ appears $m_j$ times, then $p^{(n)}(z_j)=f^{(n)}(z_j)$ for $n={0,...,m_j-1}$. Then for each $j={0,...,m-1}$,
  \begin{equation}\label{xqi}
    \left|f^{(j)}(z_0)-p^{(j)}(z_0)\right|\le C r^{m-j}\sup_{|z-z_0|=1}\left|f(z)\right|,
  \end{equation}
  here the constant $C$ depends only on $m$.
\end{lemma}

Fix $\{\lambda_n\}_{n\in \mathbb{Z}_j^-}$ to be the subspectrum of the problem $L(q,h,\alpha,\beta)$.  To prove Theorem \ref{th4.1}, we shall use the data $\beta$ and $\{\tilde{\lambda}_n\}_{n\in \mathbb{Z}_j^-}$ to construct $\tilde{q}$ and $\tilde{h}$. We agree that, if a certain
symbol $\delta$ denotes an object related to the problem $L(q,h,\alpha,\beta)$, then
$\tilde{\delta}$ will denote an analogous object related to the sequence $\{\tilde{\lambda}_n\}_{n\in \mathbb{Z}_j^-}$.
The notation $C$ may stand for different positive constants depending only on the problem $L(q,h,\alpha,\beta)$ and on the subspectrum $\{ \lambda_n \}_{n\in\mathbb{Z}_j^-}$.

By virtue of (\ref{4.1}), the sequence $\{\tilde{\lambda}_n\}_{n\in \mathbb{Z}_j^-}$ also has the asymptotics (\ref{2.7}). Consequently, $\tilde{\alpha}=\alpha$ and $\tilde{P}=P$. Put $\tilde{\beta}=\beta$ and $\tilde{\omega}=\omega$. Note that multiplicities of $\lambda_n$ and $\tilde \lambda_n$ may be distinct. However, for sufficiently small $\varepsilon > 0$, the inclusion $\mathcal{S}_j\subseteq \tilde{\mathcal{S}}_j$ holds. In particular, $\tilde{m}_n=1$ for $|n|\ge n_0$.

Denote
 \begin{equation}\label{xs1}
 \tilde{u}_{n + \nu}(t):=(it)^\nu e^{i\tilde{\lambda}_nt},\quad \tilde{w}_{n + \nu}:=-f^{(\nu)}(\tilde{\lambda}_n),
\end{equation}
for $n\in \tilde{\mathcal{S}}_j$ and $\nu=0,1,\dots,\tilde{m}_n-1$.
Consider the system of equations
\begin{equation}\label{kam}
 \frac{1}{2} \left(\tilde{M}+\alpha \tilde{N},\tilde{u}_n\right)=\tilde{w}_n,\quad n\in \mathbb{Z}_j^-,
\end{equation}
where the unknown functions $\tilde{M}(t)$ and $\tilde{N}(t)$ are respectively even and odd.

Fix $k\in(-n_0,n_0)\cap \mathcal{S}_j$, and assume that the eigenvalue $\lambda_k$  with multiplicity $m_k$ corresponds to the numbers $\{\tilde{\lambda}_n\}_{n\in M_k}$, where $M_k:=\left\{k, k + 1, \ldots, k + m_k - 1\right\}$.
Define $\tilde{\mathcal{S}}_k^j=\tilde{\mathcal{S}}_j\cap M_k$.
It is obvious that the relation (\ref{kam}) for $n\in M_k$ can be rewritten as
\begin{equation}\label{kam1}
 \frac{1}{2} \left(\tilde{M}(t)+\alpha \tilde{N}(t),(it)^{\nu}e^{i\tilde{\lambda}_nt}\right)=-f^{(\nu)}(\tilde{\lambda}_n),\quad n\in \tilde{\mathcal{S}}_k^j,\quad \nu=0,1,...,\tilde{m}_n-1.
\end{equation}
For each fixed $t\in[-a,a]$, let $E_k(t,\lambda)$, $F_k(\lambda)$ be  the unique polynomials of degree at most $m_k-1$, respectively, interpolating $e^{i\lambda t}$ and $-f(\lambda)$ and their derivatives in the usual way at the points $\{\tilde{\lambda}_n\}_{n\in M_k}$. Namely,
\begin{equation*}
E_k^{(\nu)}(t,\tilde{\lambda}_n)= (it)^{\nu}e^{i\tilde{\lambda}_nt},\quad    F_k^{(\nu)}(\tilde{\lambda}_n)=-f^{(\nu)}(\tilde{\lambda}_n),\quad  \nu=0,1,...,\tilde{m}_n-1.
\end{equation*}
 It follows from (\ref{kam1}) that
\begin{equation}\label{kam2}
 \frac{1}{2} \left(\tilde{M}(\cdot)+\alpha \tilde{N}(\cdot),E_k^{(\nu)}(\cdot,\tilde{\lambda}_n)\right)=F_k^{(\nu)}(\tilde{\lambda}_n),\quad n\in \tilde{\mathcal{S}}_k^j,\quad \nu=0,1,...,\tilde{m}_n-1.
\end{equation}
Since $E_k(t,\lambda)$, $F_k(\lambda)$ are  the polynomials of degree at most $m_k-1$, we have
\begin{equation}\label{kam3}
 \frac{1}{2} \left(\tilde{M}(\cdot)+\alpha \tilde{N}(\cdot),E_k(\cdot,{\lambda})\right)=F_k(\lambda),\quad \lambda\in \mathbb{C}.
\end{equation}
In particular, we have
\begin{equation}\label{kam4}
 \frac{1}{2} \left(\tilde{M}(\cdot)+\alpha \tilde{N}(\cdot),E_k^{(\nu)}(\cdot,{\lambda}_k)\right)=F_k^{(\nu)}({\lambda}_k),\quad \nu=0,1,...,{m}_k-1.
\end{equation}
Define the sequence $\{\tilde{\tilde{u}}_n\}_{n\in \mathbb{Z}_j^-}$ as follows
\begin{equation}\label{kam5}
\begin{split}
  &\tilde{\tilde{u}}_{n + \nu}(t)=E_n^{(\nu)}(t,{\lambda}_n),\; \tilde{\tilde{w}}_{n + \nu}=F_n^{(\nu)}({\lambda}_n),\; |n|<n_0, n\in \mathcal{S}_j, \; \nu=0,1,...,{m}_n-1,\\
  &\tilde{\tilde{u}}_n(t)=e^{i\tilde{\lambda}_nt},\quad\tilde{\tilde{w}}_n={\tilde{w}}_n,\quad |n|\ge n_0.
\end{split}
\end{equation}
Then the system (\ref{kam}) is equivalent to
\begin{equation}\label{kam6}
 \frac{1}{2} \left(\tilde{M}+\alpha \tilde{N},\tilde{\tilde{u}}_n\right)=\tilde{\tilde{w}}_n,\quad n\in \mathbb{Z}_j^-,
\end{equation}

Let us prove the following two lemmas successively.
\begin{lemma}\label{l4.1}
There exists $\varepsilon>0 $ such that, for any sequence $\{\tilde{\lambda}_n\}_{n\in \mathbb{Z}_j^-}$ satisfying \eqref{4.1}, the following estimates hold
\begin{equation}\label{4.2}
\sqrt{  \sum_{n\in \mathbb{Z}_j^-}n^2\|u_n-\tilde{\tilde{u}}_n\|_{L^2(-a,a)}^2} \le C \Lambda,
\end{equation}
\begin{equation}\label{4.3}
\sqrt{  \sum_{n\in \mathbb{Z}_j^-}|w_n-\tilde{\tilde{w}}_n|^2} \le C \Lambda.
\end{equation}
\end{lemma}
\begin{proof}
Using the Schwarz's lemma (see, e.g., \cite[p.51]{FY}), one can obtain
\begin{equation}\label{4.4}
    |e^{i\lambda_nt}|\le C,\quad |e^{i\lambda_nt}-e^{i\tilde{\lambda}_nt}|\le C|\lambda_n-\tilde{\lambda}_n|,\quad t\in [-a,a],\quad |n|\ge n_0.
\end{equation}
Substituting (\ref{4.4}) into  (\ref{s1}), we get
\begin{equation*}
  \|u_n-\tilde{\tilde{u}}_n\|_{L^2}\le C |\lambda_n-\tilde{\lambda}_n|,\quad |n|\ge n_0,
\end{equation*}
which implies
\begin{equation}\label{kam7}
 {  \sum_{|n|\ge n_0}n^2\|u_n-\tilde{\tilde{u}}_n\|_{L^2(-a,a)}^2}<C^2 \Lambda^2.
\end{equation}
Substituting (\ref{4.4}) into (\ref{x9}), we have
 \begin{equation*}
  |w_n-\tilde{\tilde{w}}_n|\le C n|\lambda_n-\tilde{\lambda}_n|, \quad |n|\ge n_0,
\end{equation*}
which implies
\begin{equation}\label{kam8}
{  \sum_{|n|\ge n_0}|w_n-\tilde{\tilde{w}}_n|^2}<C^2 \Lambda^2.
\end{equation}
Now let us consider $|n|<n_0$. By the definitions of $E_n(t,\lambda)$ and $F_n(\lambda)$, using Lemma \ref{l2.3}, we have that for each fixed $k\in(-n_0,n_0)\cap \mathcal{S}_j$,
$$
 \left|E_k^{(\nu)}(t,\lambda_k)-(it)^\nu e^{i\lambda_kt}\right|\le C\max_{n\in \tilde{\mathcal{S}}_k^j}|\tilde{\lambda}_n-\lambda_k|,\quad \nu=0,1,...,{m}_k-1,
$$
$$
 \left|F_k^{(\nu)}(\lambda_k)+f^{(\nu)}(\lambda_k)\right|\le C \max_{n\in \tilde{\mathcal{S}}_k^j}|\tilde{\lambda}_n-\lambda_k|,\quad \nu=0,1,...,{m}_k-1,
$$
for sufficient small $\varepsilon>0$. Thus
\begin{equation}\label{kam10}
 \sum_{n\in M_k} \|\tilde{\tilde{u}}_n-u_n\|_{L^2(-a,a)}\le C\max_{n\in \tilde{\mathcal{S}}_k^j}|\tilde{\lambda}_n-\lambda_k|,\quad |k|<n_0,\;k\in \mathcal{S}_j,
\end{equation}
\begin{equation}\label{kam11}
 \sum_{n\in M_k} |\tilde{\tilde{w}}_n-w_n|\le C\max_{n\in \tilde{\mathcal{S}}_k^j}|\tilde{\lambda}_n-\lambda_k|,\quad |k|<n_0,\;k\in \mathcal{S}_j.
\end{equation}
It follows that
\begin{equation}\label{kam12}
 {  \sum_{|n|<n_0}n^2\|u_n-\tilde{\tilde{u}}_n\|_{L^2(-a,a)}^2}<C^2 \Lambda^2,
\end{equation}
\begin{equation}\label{kam13}
 {  \sum_{|n|<n_0}n^2|w_n-\tilde{\tilde{w}}_n|^2}<C^2 \Lambda^2.
\end{equation}
Together with (\ref{kam7}), (\ref{kam12}) and (\ref{kam8}), (\ref{kam13}), we arrive at (\ref{4.2}) and (\ref{4.3}), respectively.
\end{proof}

\begin{lemma}\label{l4.2}
  There exists $\varepsilon> 0$ such that, for any sequence $\{ \tilde \lambda_n \}_{n \in \mathbb Z_j^-}$ satisfying \eqref{4.1},
there exists a unique pair of functions $\tilde{M}(t)$ and $ \tilde{N}(t)$ in $ L^2(-a,a)$  satisfying the relation (\ref{kam}), where $M(t)$ is odd and $N(t)$ is even.
Moreover,
\begin{equation}\label{4.6}
\|M - \tilde{M} \|_{L^2(-a,a)}+\|N - \tilde{N} \|_{L^2(-a,a)} \le C \Lambda.
\end{equation}
\end{lemma}

\begin{proof}
Using Proposition~4.1 from \cite{X} together with Lemma \ref{l4.1}, we conclude that there exists a unique function $\tilde{U}\in L^2(-a,a)$ such that  $(\tilde{U},  \tilde{\tilde{u}}_n) =  \tilde{\tilde{w}}_n$, $n \in \mathbb Z_j^-$ and $\|\tilde{U}-\frac{1}{2}({M+\alpha{N}}) \| \le C \Lambda$. Denote
$${\tilde{M}(t)}:={\tilde{U}(t)+\tilde{U}(-t)},\quad {\tilde{N}(t)}:=\frac{\tilde{U}(t)-\tilde{U}(-t)}{\alpha}.$$
Then $\tilde{U}(t)=\frac{1}{2}\left({\tilde{M}(t)}+\alpha{ \tilde{N}(t)}\right)$, and $\tilde{M}(t)$ is even and $\tilde{N}(t)$ is odd. Thus the system (\ref{kam6}) is satisfied, which is equivalent to (\ref{kam}). By a direct calculation, we can obtain (\ref{4.6}).
\end{proof}

Define the functions $\tilde{\Delta}(\lambda)$, $\tilde{K}_x(a,t)$ and $\tilde{K}_t(a,t)$ with the functions $\tilde{M}(t)$ and $\tilde{N}(t)$:
\begin{equation}\label{4.7}
 \tilde{ \Delta}(\lambda)=f(\lambda)+\frac{1}{2}\int_{-a}^a[\overline{\tilde{M}(t)}+\alpha \overline{\tilde{N}(t)}]e^{i\lambda t}dt,
\end{equation}
\begin{equation}\label{x12}
\tilde{K}_x(a,t):=\overline{\tilde{M}(t)}-{\beta}\int_t^a\overline{\tilde{N}(s)}ds-\beta \omega,\quad \tilde{K}_t(a,t):=-\overline{\tilde{N}(t)},\quad t\in[0,a].
\end{equation}
Clearly, $\tilde{K}_x(a,t)$ is even, and $\tilde{K}_t(a,t)$ is odd.
It follows from (\ref{kam}) and (\ref{4.7}) that $\{\tilde{\lambda}_n\}_{n\in\mathbb{Z}_j^-}$ (with multiplicities) are the zeros of $\tilde{\Delta}(\lambda)$.
By the Schwarz's inequality, we calculate
\begin{equation*}
  \int_0^a \left|\int_t^a(\tilde{N}(s)-N(s))ds\right|^2dt\le a \left(\int_0^a \left|\tilde{N}(s)-N(s))\right|ds\right)^2\le a^2\|\tilde{N}-N\|_{L^2(0,a)}^2.
\end{equation*}
It follows from (\ref{4.6}) and (\ref{x12}) that
\begin{equation}\label{4.10}
\left\|\tilde{K}_x(a,\cdot)-K_x(a,\cdot)\right\|_{L^2(0,a)}+\left\|\tilde{K}_t(a,\cdot)-K_t(a,\cdot)\right\|_{L^2(0,a)}\le C\Lambda.
\end{equation}
Using Theorem \ref{cau}  and (\ref{4.10}), we get that there exist the unique pair of $\tilde{q}$ and $\tilde{h}$ such that  $\{\tilde{K}_x(a,t), \tilde{K}_t(a,t),\omega\}$ are the corresponding Cauchy data. Moreover,
$$
\|q-\tilde{q}\|_{L^{2}(0, a)} \leq C \Lambda,\quad |\tilde{h}-h|\leq C \Lambda.
$$

Define the functions
\begin{equation}\label{2.1q}
  \tilde{\varphi}(a,\lambda)=\cos (\lambda a)+\frac{\omega\sin(\lambda a)}{\lambda}-\int_{0}^a\tilde{K}_t(a,t)\frac{\sin (\lambda t)}{\lambda}dt,
\end{equation}
\begin{equation}\label{2.2q}
  \tilde{\varphi}'(a,\lambda)=-\lambda\sin (\lambda a)+{\omega\cos(\lambda a)}+\int_{0}^a\tilde{K}_x(a,t)\cos (\lambda t)dt.
\end{equation}
Using (\ref{4.7}), (\ref{x12}), (\ref{2.1q}), (\ref{2.2q}), and (\ref{x6s}), we obtain that the function $\tilde{\Delta}(\lambda)$ constructed in \eqref{4.7}
 has the expression
 \begin{equation}\label{a4.9}
\tilde{ \Delta}(\lambda)=\tilde{\varphi}'(a,\lambda)+(i\lambda \alpha+\beta)\tilde{\varphi}(a,\lambda).
\end{equation}
The proof of Theorem \ref{th4.1} is complete.
\\[2mm]

\noindent {\bf Acknowledgments.}
The author Xu was supported  by the National Natural Science Foundation of China (11901304). The author Bondarenko was supported by Grants 20-31-70005 and 19-01-00102 of the Russian Foundation for Basic Research.

\end{document}